\documentclass{amsart}

\usepackage[dvipsnames]{xcolor}

\usepackage{lscape}
\usepackage{pst-plot}

\usepackage{amsrefs,amsmath,amsthm,amssymb}            
\usepackage{enumerate,lscape,calc}
\usepackage[all]{xy}   
\usepackage{longtable,booktabs}
\usepackage{pstricks,color}

\usepackage{tikz-cd}

\usepackage{caption}

\newtheorem{thm}{Theorem}%[section]
\newtheorem{prop}[thm]{Proposition}
\newtheorem{cor}[thm]{Corollary}
\newtheorem{lemma}[thm]{Lemma}
\newtheorem{conj}{Conjecture} 

\theoremstyle{definition}
\newtheorem{defn}[thm]{Definition}
\newtheorem{ex}[thm]{Example}

\theoremstyle{remark}
\newtheorem{remark}[thm]{Remark}

\numberwithin{equation}{section}

\usepackage{graphicx}
\usepackage[final]{pdfpages}
\begin{document}

\title{THE $\mathbb{C}$-MOTIVIC WEDGE SUBALGEBRA}

\author{Hieu Thai}
\address{Department of Mathematics\\
Wayne State University\\
Detroit, MI 48202, USA}
\email{hieu.thai@wayne.edu}
\thanks{}



\date{}

\begin{abstract}We describe some regular behavior in the motivic wedge, which is a subalgebra of the cohomology Ext$_{\mathbf{A}}(\mathbb{M}_2,\mathbb{M}_2)$ of the $\mathbb{C}$-motivic Steenrod algebra. The key tool is to compare motivic computations to classical computations, to Ext$_{\mathbf{A}(2)}(\mathbb{M}_2,\mathbb{M}_2)$, or to $h_1$-localization of Ext$_{\mathbf{A}}(\mathbb{M}_2,\mathbb{M}_2)$. 
	
We also give a conjecture on the behavior of the family $e_0^tg^k$ in Ext$_{\mathbf{A}}(\mathbb{M}_2,\mathbb{M}_2)$ which raises naturally from the study of the motivic wedge. 
\end{abstract}

\maketitle

\section{Introduction}
Computing the stable homotopy groups of the sphere spectrum is one of the most important problems of stable homotopy theory. Focusing on the 2-complete stable homotopy groups instead of the integral homotopy groups, the Adams spectral sequence appears to be one of the most effective tools to compute the homotopy groups. The spectral sequence has been studied by F. Adams, M. Mahowald, P. May, M. Tangora and others \cite{M-thesis} \cite{T-thesis} \cite{diff-A} \cite{diff-A-2} \cite{diff-A-B}.

In 1999, Voevodsky and Morel introduced motivic homotopy theory \cite{VM-scheme}. One of its consequences is the realization that almost any object studied in classical algebraic topology could be given a motivic analog. In particular, we can define the motivic Steenrod algebra $\mathbf{A}$ \cite{Vod-S1}, the motivic stable homotopy groups of spheres \cite{VM-scheme} and the motivic Adams spectral sequence \cite{I-MASS}. In the motivic perspective, there are many more non-zero classes in the motivic Adams spectral sequence, which allows the detection of otherwise elusive phenomena. Also, the additional motivic weight grading can eliminate possibilities which appear plausible in the classical perspective.  

Let $\mathbb{M}_2$ denote the motivic cohomology of a point, which is isomorphic to $\mathbb{F}_2[\tau]$ where $\tau$ has bidegree $(0,1)$ \cite{Vod-Z2}. The motivic Steenrod algebra $\mathbf{A}$ is the $\mathbb{M}_2$-algebra generated by elements $\mathrm{Sq}^{2k}$ and $\mathrm{Sq}^{2k-1}$ for all $k\ge 1$, of bidegrees $(2k,k)$ and $(2k-1,k-1)$ respectively, subject to Adem relations \cite{Vod-S1} \cite{Vod-S2}. Let Ext$_{\mathbf{A}}(\mathbb{M}_2,\mathbb{M}_2)$ denote the cohomology of the motivic Steenrod algebra. To run the motivic Adams spectral sequence, one begins with Ext$_{\mathbf{A}}(\mathbb{M}_2,\mathbb{M}_2)$. The cohomology Ext$_{\mathbf{A}}(\mathbb{M}_2,\mathbb{M}_2)$ has an $\mathbb{M}_2$-algebra structure. Inverting $\tau$ in Ext$_\mathbf{A}(\mathbb{M}_2,\mathbb{M}_2)$ gives the cohomology Ext$_{\mathbf{A}_{cl}}(\mathbb{F}_2,\mathbb{F}_2)$ of the classical Steenrod algebra $\mathbf{A}_{cl}$ \cite{I-stem}. Given a classical element, there are many corresponding motivic elements. We typically want to find the corresponding element with the highest weight. For example, the classical element $g$ corresponds to the motivic elements $\tau^k g$ for all $k \geq 1$. The element $\tau g$ has weight $11$, but there is no motivic element of weight $12$ that corresponds to the classical element $g$.

The algebra Ext$_{\mathbf{A}}(\mathbb{M}_2,\mathbb{M}_2)$ is infinitely generated and irregular. A natural approach is to look for systematic phenomena in Ext$_{\mathbf{A}}(\mathbb{M}_2,\mathbb{M}_2)$. One potential candidate is the wedge family in Ext$_{\mathbf{A}}(\mathbb{M}_2,\mathbb{M}_2)$.

The classical wedge family was studied by M. Mahowald and M. Tangora \cite{Maho-Tango-wedge}. It is a subset of the cohomology Ext$_{\mathbf{A}_{\mathrm{cl}}}(\mathbb{F}_2,\mathbb{F}_2)$ of the classical Steenrod algebra, consisting of non-zero elements $P^ig^j \lambda$ and $g^j t$ in which $\lambda$ is in $\boldsymbol\Lambda$, $t$ is in $\mathbf{T}$, $i\ge 0$ and $j\ge 0$. The sets $\boldsymbol\Lambda $ and $\mathbf{T}$ are specific subsets of Ext$_{\mathbf{A}_{\mathrm{cl}}}(\mathbb{F}_2,\mathbb{F}_2)$. The wedge family gives an infinite  wedge-shaped diagram inside the cohomology of the classical Steenrod algebra, which fills out an angle with vertex at $g^2$ in degree (40,8) (i.e. $g^2$ has stem 40 and Adams filtration 8), bounded above by the line $f=\frac{1}{2}s-12$, parallel to the Adams edge \cite{Adams-ped}, and bounded below by the line $s=5f$, in which $f$ is the Adams filtration and $s$ is the stem. The wedge family is a large piece of Ext$_{\mathbf{A}_{\mathrm{cl}}}(\mathbb{F}_2,\mathbb{F}_2)$ which is regular, of considerable size and easy to understand.  

Using this idea we build the motivic version of the wedge. However, it appears to be more complicated than the classical one. The highest weights of the motivic wedge elements follow a somewhat irregular pattern. We will discuss this irregularity in more detail later.

Let $\mathbf{A}(2)$ denote the $\mathbb{M}_2$-subalgebra of $\mathbf{A}$ generated by $\mathrm{Sq}^1, \mathrm{Sq}^2$ and $\mathrm{Sq}^4$. Let Ext$_{\mathbf{A}(2)}(\mathbb{M}_2,\mathbb{M}_2)$ denote the cohomology of $\mathbf{A}(2)$. The finitely generated algebra Ext$_{\mathbf{A}(2)}(\mathbb{M}_2,\mathbb{M}_2)$ is fully understood by \cite{I-A2}. We use a new technique of comparison to Ext$_{\mathbf{A}(2)}(\mathbb{M}_2,\mathbb{M}_2)$ which makes the proof of the non-triviality of the wedge elements easy. We consider the ring homomorphism $\phi$ from Ext$_{\mathbf{A}}(\mathbb{M}_2,\mathbb{M}_2)$ to Ext$_{\mathbf{A}(2)}(\mathbb{M}_2,\mathbb{M}_2)$ induced by the inclusion from $\mathbf{A}(2)$ to $\mathbf{A}$. We use the map $\phi$ to detect structure in Ext$_{\mathbf{A}}(\mathbb{M}_2,\mathbb{M}_2)$. Most of the elements studied in this article have non-zero images via $\phi $ \cite{I-A2}. Therefore, they are all non-trivial elements in Ext$_{\mathbf{A}}(\mathbb{M}_2,\mathbb{M}_2)$. 

We define set-valued operations $\mathbf{P}$ and $\mathbf{g}$ on Ext$_{\mathbf{A}}(\mathbb{M}_2,\mathbb{M}_2)$. Classically, $g$ is an element of the cohomology of the classical Steenrod algebra. However, this is not true motivically. Rather, $\tau g$ is an element in Ext$_{\mathbf{A}}(\mathbb{M}_2,\mathbb{M}_2)$, while $g$ itself does not survive the motivic May spectral sequence. Consequently, multiplication by $g$ does not make sense motivically. Also, $P$ is not an element in Ext$_\mathbf{A}(\mathbb{M}_2,\mathbb{M}_2)$ either. We instead consider the set-valued operations $\mathbf{P}$ and $\mathbf{g}$ whose actions can be seen as multiplications by $P$ and $g$ in Ext$_{\mathbf{A}(2)}(\mathbb{M}_2,\mathbb{M}_2)$ respectively. 

For any $\lambda$ in Ext$_{\mathbf{A}}(\mathbb{M}_2,\mathbb{M}_2)$, $i \ge 0$ and $j\ge 0$, let $\mathbf{P}^i\mathbf{g}^j \lambda$ be the set consisting of all elements $x$ in Ext$_\mathbf{A}(\mathbb{M}_2,\mathbb{M}_2)$ such that $\phi(x) = P^ig^j \phi(\lambda)$ in Ext$_{\mathbf{A}(2)}(\mathbb{M}_2,\mathbb{M}_2)$. 

We define the wedge family via the actions of $\mathbf{P}$ and $\mathbf{g}$. The wedge is the set consisting of all elements in $\mathbf{P}^i\mathbf{g}^j \lambda$ with $i \ge 0$ and $j\ge 0$, where $\lambda$ is contained in a specific 16-element subset $\boldsymbol\Lambda$ of Ext$_{\mathbf{A}}(\mathbb{M}_2,\mathbb{M}_2)$ to be defined in Table \ref{lambda}.

The motivic wedge family takes the same position and same shape as the classical one (Figure \ref{figure}). However the vertex of the motivic wedge is at $\tau g^2$ in degree $(40,8,23)$ having weight 23. Note that $g^2$ in degree $(40,8,24)$ does not survive the motivic May spectral sequence \cite{I-stem}. Our main result, Theorem \ref{def-wedge}, states that the subsets $\mathbf{P}^i\mathbf{g}^j\lambda$ are non-empty and consist of non-zero elements for all $\lambda$ in $\boldsymbol\Lambda$. 

However, our main result is not optimal, in the sense that there exist elements of weight greater than the weight of elements in $\mathbf{P}^i \mathbf{g}^j \lambda$ for some values of $i$, $j$, and $\lambda$. Some such elements are listed in Table \ref{opt}.

We can not even conjecture the optimal result in general.  However, we know a bit more about elements in the set $e_0^t \mathbf{g}^k$ for $t \geq 0$ and $k \geq 0$, which are part of the wedge.  We will show that $\tau e_0^t \mathbf{g}^k$ is non-empty for all $t \geq 0$ and $k \geq 0$.  We do not know whether $e_0^t \mathbf{g}^k$ is non-empty in general, but we make the following conjecture.

\begin{conj}
The set $e_0^t\mathbf{g}^k$ is non-empty if and only if $k=\sum ^{t}_{i=0}2^{n_i}-t$ for some integer $n_i \ge 1$. \end{conj}
The conjecture holds if and only if $e_0\mathbf{g}^{2^n-1}$ is non-empty for all $n\ge 1$, since $$e_0^t\mathbf{g}^k \supseteq e_0\mathbf{g}^{2^{n_1}-1}  \cdots  e_0\mathbf{g}^{2^{n_t}-1}.$$

By explicit computations we know that $e_0, e_0\mathbf{g}$ and $e_0\mathbf{g}^3$ are non-empty and $e_0\mathbf{g}^2$ and $e_0\mathbf{g}^4$ are empty \cite{I-stem}. This means that the subsets $e_0 \mathbf{g}^k$ are non-empty sometimes but empty other times. The analogous classical question is trivial, since $e_0^t g^k$ is a product of $e_0^t$ and $g^k$.

\subsection{Organization}
The article contains four sections. In the second section, we recall techniques of comparison to classical computations, to Ext$_{\mathbf{A}(2)}(\mathbb{M}_2,\mathbb{M}_2)$ and to $h_1$-localization of Ext$_\mathbf{A}(\mathbb{M}_2,\mathbb{M}_2)$. In the third section, we introduce our main result on the motivic wedge. In the last section, we study the family $e_0^t\mathbf{g}^k$ in Ext$_\mathbf{A}(\mathbb{M}_2,\mathbb{M}_2)$ and give a conjecture on its behavior. 
\subsection{Acknowledgements}
I would like to acknowledge the invaluable assistance of Daniel Isaksen. I also want to thank Robert Bruner, Andrew Salch and Bogdan Gheorghe for useful discussions. 

\subsection{Notation}
We will use the following notation. 
\begin{enumerate}
	\item $\mathbb{M}_2=\mathbb{F}_2[\tau]$ is the mod 2 motivic cohomology of a point, where $\tau$ has bidegree (0,1).
	\item $\mathbf{A}$ is the mod 2 motivic Steenrod algebra over $\mathbb{C}$.
	\item $\mathbf{A}(2)$ is the $\mathbb{M}_2$-subalgebra of $\mathbf{A}$ generated by $\mathrm{Sq}^1, \mathrm{Sq}^2$ and $\mathrm{Sq}^4$.
	\item $\mathbf{A}_{\mathrm{cl}}$ is the mod 2 classical Steenrod algebra. 
	\item Ext is the trigraded ring Ext$_{\mathbf{A}}(\mathbb{M}_2,\mathbb{M}_2)$, the cohomology of the motivic Steenrod algebra.
	\item Ext$_{\mathbf{A}(2)}$ is the trigraded ring Ext$_{\mathbf{A}(2)}(\mathbb{M}_2,\mathbb{M}_2)$, the cohomology of the $\mathbb{M}_2$-subalgebra of $\mathbf{A}$ generated by $\mathrm{Sq}^1, \mathrm{Sq}^2$ and $\mathrm{Sq}^4$.
	\item Ext$_{\mathrm{cl}}$ is the bigraded ring Ext$_{\mathbf{A}_{\mathrm{cl}}}(\mathbb{F}_2,\mathbb{F}_2)$, the cohomology of the classical Steenrod algebra.
	\item We use the notation of \cite{more-stem} for elements in Ext. 
	\item We use the notation of \cite{I-A2} for elements in Ext$_{\mathbf{A}(2)}$, except that we use $a$ and $n$ instead of $\alpha$ and $\nu$ respectively. 
	\item An element $x$ in Ext has degree of the form $(s,f,w)$ where:
	\begin{enumerate}\item $f$ is the Adams filtration, i.e., the homological degree.
		\item $s+f$ is the internal degree, i.e., corresponds to the first coordinate in the bidegrees of $\mathbf{A}$.
		\item $s$ is the stem, i.e., the internal degree minus the Adams filtration.
		\item $w$ is the motivic weight.
	\end{enumerate}
	\item The Chow degree of an element of degree $(s,f,w)$ is $s+f-2w$. 
	\item The coweight of an element of degree $(s,f,w)$ is $s-w$.
	
\end{enumerate}

\section{Comparison Criteria}

We know Ext$_{\mathrm{cl}}$ \cite{M-thesis} \cite{T-thesis} and Ext$_{\mathbf{A}(2)}$ \cite{I-A2} quite well. Computations in Ext can be studied via the relation with Ext$_{\mathrm{cl}}$ and Ext$_{\mathbf{A}(2)}$ in certain cases.

The following theorem plays a key role in comparing the motivic and the classical computations, saying that they become the same after inverting $\tau$.
\begin{thm}\cite{I-MASS} There is an isomorphism of rings
	$$\mathrm{Ext}\otimes _{\mathbb{M}_2} \mathbb{M}_2[\tau ^{-1}] \cong \mathrm{Ext}_{{\mathrm{cl}}} \otimes _{\mathbb{F}_2} \mathbb{F}_2[\tau, \tau ^{-1}].$$
\end{thm}

Furthermore, the part of Ext at Chow degree 0 is isomorphic to Ext$_{\mathrm{cl}}$.
\begin{thm}\label{Chow}\cite{I-stem} There is an isomorphism from $\mathrm{Ext}_{\mathrm{cl}}$ to the subalgebra of $\mathrm{Ext}$ consisting of elements in degrees $(s,f,w)$ such that $s+f-2w=0$. This isomorphism takes classical elements of degrees $(s,f)$ to motivic elements of degrees $(2s+f,f,s+f)$, and it preserves all higher structure including products, squaring operations, and Massey products.
\end{thm} 
In other words, $$\mathrm{Ext}|_{s+f-2w=0} \cong \mathrm{Ext}_{\mathrm{cl}}.$$

The inclusion $\mathbf{A}(2) \hookrightarrow \mathbf{A}$ induces a homomorphism $\phi :$ Ext $ \rightarrow $ Ext$_{\mathbf{A}(2)}$ which allows us to detect some structure in Ext via Ext$_{\mathbf{A}(2)}$.
We emphasize that Ext$_{\mathbf{A}(2)}$ is described completely in \cite{I-A2}. Table \ref{table-phi} gives some values of $\phi$ that we will need.
\begin{table}[h]
	\caption{Some values of the map $\phi: \mathrm{Ext} \longrightarrow \mathrm{Ext}_{\mathbf{A}(2)}$.}
	\label{table-phi}
	\begin{tabular}{c|c|c}
		
		Ext & Ext$_{\mathbf{A}(2)}$ & $(s,f,w)$\\
		\hline
		$i$ & $Pn$ & $(23,7,12)$  \\   
		
		$k$ & $dn$ & $(29,7,16)$ \\
		
		$r$ & $n^2$ & $(30,6,16)$  \\
	
		$m$  & $ng$ & $(35,7,20)$ \\
		
			$ \Delta h_1 d_0$ & $ \Delta h_1 \cdot d_0$ & $(39,9,21)$\\
		
		$\tau g^2$ & $\tau \cdot g^2$ & $(40,8,23)$\\
			
		$\tau \Delta h_1 g$ & $\tau \cdot \Delta h_1 \cdot g$ & $(45,9,24)$\\
		
			$h_2 g^j$ & $h_2 \cdot g^j$ & ($20j+3,4j+1,12j+2$) \\
		
				$P^id_0$  & $P^i\cdot d_0$  & ($8i+14,4i+4,4i+8$)\\
		
		$P^ie_0$ & $P^i \cdot e_0$ & ($8i+17,4i+4,4i+10$)
	
	\end{tabular}
\end{table}

\begin{remark}In some cases, $\phi(x)$ is decomposable in Ext$_{\mathbf{A}(2)}$ when $x$ is indecomposable in Ext. For example, the element $\Delta h_1 d_0$ in Ext is indecomposable but $\phi(\Delta h_1 d_0)=\Delta h_1 \cdot d_0$ is the product of $\Delta h_1$ and $d_0$ in Ext$_{\mathbf{A}(2)}$. 
	\end{remark}

If we invert $h_1$ on Ext, then Ext$[h_1^{-1}]$ becomes simpler. We can use Ext$[h_1^{-1}]$ to detect some structure in Ext. The following theorems describe Ext$[h_1^{-1}]$ and Ext$_{\mathbf{A}(2)}[h_1^{-1}]$.
\begin{thm}\cite{G-I-local}*{Theorem 1.1}\label{thm:h1-local} 
	The $h_1$-localization $\mathrm{Ext}[h_1^{-1}]$ is a polynomial algebra over $\mathbb{F}_2[h_1^{\pm 1}]$ on generators $v_1^4$ and $v_n$ for $n \ge 2$, where:
	\begin{itemize}
		\item	
		the element $v_1^4$ has degree $(8,4,4)$.
		\item
		the element $v_n$ has degree $(2^{n+1}-2,1,2^n-1)$.
	\end{itemize}
\end{thm}

\begin{thm}\cite{G-I-local}The $h_1$-localization $\mathrm{Ext}_{\mathbf{A}(2)} [h_1^{-1}]$ is a polynomial algebra
	$$\mathrm{Ext}_{\mathbf{A}(2)} [h_1^{-1}] \cong \mathbb{F}_2[h_1^{\pm 1},a_1,v_1^4,v_2]$$
	in which $a_1$ has degree $(11,3,7)$; $v_1^4$ has degree $(8,4,4)$; and $v_2$ has degree $(6,1,3)$.	
	\end{thm}

We can use $h_1$-localization to prove the non-existence of certain elements $x$ in Ext. We use the May spectral sequence analysis of Ext$[h_1^{-1}]$ to determine the localization map
$$L: \mathrm{Ext} \longrightarrow \mathrm{Ext}[h_1^{-1}]$$
in a range \cite{G-I-local}. Some values of $L$ are given in Table \ref{L-map}.
\begin{table}[h] 
	\caption{Some values of the localization map $L: \mathrm{Ext} \rightarrow \mathrm{Ext}[h_1^{-1}]$ \cite{G-I-local} \label{L-map}}\begin{tabular}{c|c}
		
		$x$ & $L(x)$ \\
		\hline
		$P^k h_1$  & $h_1v_1^{4k}$  \\
		
		$P^k d_0$ & $h_1^2 v_1^{4k} v_2^2$   \\
		
		$P^k e_0$ & $h_1^3 v_1^{4k} v_3$   \\   
		
		$e_0 g$ & $h_1^7 v_4$  \\
	
\end{tabular} \end{table}

There is also a localization map $L:$ Ext$_{\mathbf{A}(2)} \longrightarrow$ Ext$_{\mathbf{A}(2)}[h_1^{-1}]$.
The following diagram is commutative.
\[
\begin{tikzcd}
\mathrm{Ext}  \ar{r}{\phi} \arrow{d}{L} & 
\mathrm{Ext}_{\mathbf{A}(2)} \arrow{d}{L}  \\
\mathrm{Ext}[h_1^{-1}] \ar{r}{\phi} & \mathrm{Ext}_{\mathbf{A}(2)}[h_1^{-1}] 
\end{tikzcd}
\] 

\begin{defn}Let $t$ be a non-negative integer. We define $\alpha (t)$ to be the number of 1's in the binary expansion of $t$.
\end{defn}
\begin{lemma}\label{lem:binary}
	Let $t, k$ and $s$ be non-negative integers, $s\ge 1$. We have
	\begin{itemize}
		\item
		$\alpha (t) \le t$.
		\item
		$\alpha (t+k) \le \alpha (t) + \alpha(k)$.
		\item $\alpha (2^st) = \alpha (t).$
	\end{itemize}
\end{lemma}
\begin{proof}Suppose that $t=\sum_{i=1}^{n} 2^{m_i}$ in which $m_i \ge 0$ and $m_i \neq m_j$ if $i\neq j$. Consequently, $\alpha (t)=n$. Since $t=\sum_{i=1}^{n} 2^{m_i}\ge 1\cdot n=n$, we obtain the first inequality.
	
	With the above $t$ we suppose further that $k=\sum_{i=1}^{p} 2^{q_i}$ in which $q_i \ge 0$ and $q_i \neq q_j$ if $i\neq j$. Consequently, $\alpha (k)=p$. We have
	
	$$t+k = \sum_{i=1}^{n} 2^{m_i} + \sum_{j=1}^{p} 2^{q_j} $$
	where the right hand side has $n+p$ powers of 2. If there is no pair $(m_i,q_j)$ such that $m_i=q_j$, then $\alpha (t+k)=n+p = \alpha (t) + \alpha(k)$. If there exists at least one pair $(m_i,q_j)$ such that $m_i = q_j = c$, since $2^{m_i} + 2^{q_j} = 2^{c+1}$ we have $\alpha (t+k) < n+p = \alpha (t) + \alpha(k)$. Therefore, $\alpha(t+k)\le \alpha(t)+\alpha(k)$.
	
	The last identity can be proven by the observation that if $t = \sum_{i=1}^{n} 2^{m_i}$, then $2^st = \sum_{i=1}^{n} 2^{m_i+s}$. 
\end{proof}

\begin{lemma}\label{local2}The image of 
\begin{equation*}	\phi: \mathrm{Ext}[h_1^{-1}] \longrightarrow  \mathrm{Ext}_{\mathbf{A}(2)} [h_1^{-1}] \end{equation*}
 is spanned by the monomials $h_1^d v_1^{4a} v_2^b a_1^c$ where $a, b$ and $c$ are non-negative integers for which $\alpha(b+c) \le b$ and $d$ is an integer.
\end{lemma}
\begin{proof}Denote by $\mathcal{G}$ the $\mathbb{M}_2$-submodule of Ext$_{\mathbf{A}(2)} [h_1^{-1}]$ spanned by the monomials $h_1^d v_1^{4a} v_2^b a_1^c$ where $a, b$ and $c$ are non-negative integers for which $\alpha(b+c) \le b$ and $d$ is an integer.
	
	The map $\phi$ takes $v_1^4$ and $v_n$ to $v_1^4$ and $h_1^{-3(2^{n-2}-1)}a_1^{2^{n-2}-1}v_2$ respectively for all $n$ \cite{G-I-local}. 
	Consequently
	$$\phi: v_1^{4a}\prod _{j\in J} v_j \longmapsto h_1^{-3\sum_{j\in J}(2^{j-2}-1)} v_1^{4a}v_2^m a_1^{\sum_{j\in J}(2^{j-2}-1)}$$
	 in which $J$ is a sequence $(j_1,\ldots,j_m)$ of length $m$ such that $j_k \ge 2$ (repeats are allowed).
	
	Consequently, the image of $\phi$ equals the $\mathbb{M}_2$-submodule $\mathcal{H}$ of Ext$_{\mathbf{A}(2)} [h_1^{-1}]$ spanned by the monomials of the form $h_1^d v_1^{4a}v_2^m a_1^{\sum_{j\in J}(2^{j-2}-1)}$ in which $J$ is a sequence $(j_1,\ldots,j_m)$ of length $m$ such that $j_k \ge 2$ (repeats are allowed).
	
	Since $J$ has length $m$, $$\alpha(m+\sum_{j\in J}(2^{j-2}-1)) = \alpha (\sum_{j\in J}2^{j-2}) \le m.$$ As a result, $\mathcal{H}$ is contained in $\mathcal{G}$.
	
	Conversely, for any monomial $h_1^d v_1^{4a} v_2^b a_1^c$ for which $\alpha(b+c) \le b$, we can suppose that $b+c = \sum_{j\in J}2^{j} $ where $J$ is a sequence $(j_1,\ldots,j_r)$ of length $r\le b$ such that $j_k \ge 0$ for $k$ in $\{1,\ldots,r\}$. By replacing $2^j$ by $2^{j-1}+2^{j-1}$ as necessary, we can rewrite $b+c$ as
	
	$$ b+c = \sum_{i\in I}2^{i}$$  
	where $I$ is a sequence $(i_1,\ldots,i_b)$ of length $ b$ such that $i_k \ge 0$ for $k$ in $\{1,\ldots,b\}$. Then 
	$$c = \sum_{i\in I}2^{i} - b = \sum_{i\in I}(2^{i}-1). $$
	This shows that $\mathcal{G}$ is contained in $\mathcal{H}$.
\end{proof}

\begin{remark}The value of $\phi(v_n)$ is given in \cite{G-I-local} for $n \leq 6$.  In \cite{G-I-local}, it is also explained that the value of $\phi(v_n)$ for all $n$ can be deduced from the existence of a ``motivic modular forms" spectrum. Such a spectrum has recently been constructed \cite{GKIR}, so that the values of $\phi(v_n)$ are now known for all $n$. 
	\end{remark}

\section{The wedge family}
We recall that the inclusion $\mathbf{A}(2)\hookrightarrow \mathbf{A}$ induces a homomorphism of algebras $\phi: \mathrm{Ext} \rightarrow \mathrm{Ext}_{\mathbf{A}(2)}$.
\begin{defn}\label{def:Pg}For any $\lambda$ in  Ext, $i \ge 0$ and $j\ge 0$, $\mathbf{P}^i\mathbf{g}^j \lambda$ is the set which consists of all elements $x$ of Ext such that $\phi (x) = P^ig^j \phi(\lambda)$ in Ext$_{\mathbf{A}(2)}$.
\end{defn}

\begin{ex}
The set $\mathbf{g}^2r$ contains $m^2$ because $\phi (m^2) = g^2 n^2 = g^2 \phi(r).$
\end{ex}

\begin{remark}We differentiate $\mathbf{P}$ and $\mathbf{g}$ with $P$ and $g$. By the bold $\mathbf{P}$ and $\mathbf{g}$ we mean set-valued operations from Ext to Ext. Remember that $\mathbf{P}$ and $ \mathbf{g}$ do not exist in Ext as elements. By $P$ and $g$, we mean elements in Ext$_{\mathbf{A}(2)}$. 
\end{remark}

%\begin{remark}The sets $\mathbf{P}^i\mathbf{g}^j\lambda$ can have more than one element. For example, $\mathbf{P}^0\mathbf{g}^0(\tau d_0l)$ consists of $\tau d_0l$ and $\Delta c_0 d_0$. The set $\mathbf{P}^0\mathbf{g}^0(\tau^3 h_0 h_2^2g_2^2)$ consists of $\tau^3 h_0 h_2^2g_2^2$, $\tau d_0l$ and $\Delta c_0d_0$. Consequently, $\mathbf{P}^0\mathbf{g}^0(\tau d_0l)$ is a proper subset of $\mathbf{P}^0\mathbf{g}^0(\tau^3 h_0 h_2^2g_2^2)$. 
%\end{remark}	

	\begin{remark}
	More specifically, the sets $\mathbf{P}^i\mathbf{g}^j\lambda$ can be seen as cosets with respect to the subgroup $$I=\{z:\phi(z)=0\}$$ of Ext. 
	
\end{remark}
\begin{remark}We sometimes write the symbols $\mathbf{P}$ and $\mathbf{g}$ in a different order for consistency with standard notation. For example:
	\begin{itemize}
\item	By $e_0\mathbf{g}^2$ we mean $\mathbf{g}^2(e_0)$. The set $e_0\mathbf{g}^2$ is empty (See Corollary \ref{cor:empty}). 
\item 	By $\tau \Delta h_1 \mathbf{g}^{j+1}$ we mean the set $\mathbf{g}^{j} (\tau \Delta h_1 g)$. 
\item By $\tau \mathbf{P}^i\mathbf{g}^{j+1}$ we mean $\mathbf{P}^i\mathbf{g}^{j+1}(\tau)$.
\item The same convention is applied for $\tau \mathbf{g}^k$, $\tau e_0^t \mathbf{g}^k$ and many others.
	\end{itemize}
\end{remark}

\begin{remark}From Definition \ref{def:Pg} we have $\mathbf{P}^i\mathbf{g}^jx \cdot \mathbf{P}^a\mathbf{g}^b y \subseteq \mathbf{P}^{i+a}\mathbf{g}^{j+b}xy.$ However, the inverse inclusion is not correct generally. For example, by low dimension calculation \cite{I-stem} we have$$e_0 \cdot \tau^2 \mathbf{g} =\{e_0\}\{\tau^2 g\} \varsubsetneq \tau^2 e_0 \mathbf{g} = \{\tau^2 e_0g, \tau^2e_0g+h_0^3x\}.$$
	\end{remark}
\begin{defn}
We define $\boldsymbol\Lambda $ to be the following sixteen elements of Ext.
\begin{table}[h]
	\caption{Sixteen elements of the set $\boldsymbol\Lambda \label{lambda}$}
	\begin{tabular}{c|cc|c}
		
		element & $(s,f,w)$ & element & $(s,f,w)$ \\
		\hline 
		$\tau g^2$  & (40,8,23)  & $d_0r$ & (44,10,24)  \\
		
		$\tau \Delta h_1 g$ & (45,9,24) & $d_0m$ & (49,11,28)  \\
		
		$gr$ & (50,10,28)  & $\tau e_0^2g$ & (54,12,31)  \\
		
		$gm$ & (55,11,32) &  $\tau \Delta h_1 e_0^2$ & (59,13,32)	   \\

		$\tau \Delta h_1e_0$ & (42,9,22)  & $d_0l$ & (46,11,26)  \\
		$e_0 r$ & (47,10,26) &  $\tau e_0^3$ & (51,12,29)  \\
		$e_0m$ & (52,11,30) & 	$\tau \Delta h_1 d_0 e_0$ & (56,13,30)   \\
		$\tau e_0 g^2$ & (57,12,33) &  	$d_0 e_0 r$ & (61,14,34)  \\
		
\end{tabular} \end{table}
\end{defn}
The following theorem is our main result.
\begin{thm}\label{def-wedge}For any $\lambda$ in $\boldsymbol\Lambda$, $i \ge 0$ and $j\ge 0$, the set
	$\mathbf{P}^i \mathbf{g}^j\lambda$ is non-empty and consists of non-zero elements.
\end{thm}
Combining all elements of $\mathbf{P}^i\mathbf{g}^j\lambda$ with $i\ge 0$, $j\ge 0$ and $\lambda$ in $\boldsymbol\Lambda$, we obtain an infinite wedge-shaped diagram, filling out the angle with vertex at $\tau g^2$ in degree  $(40,8,23)$, bounded above by the line $f=\frac{1}{2}s-12$ parallel to the Adams edge \cite{Adams-mot}, and bounded below by the line $s=5f$ in Ext (Figure \ref{figure}). 

We need a couple of preliminary results before proving Theorem \ref{def-wedge}.
\begin{lemma}\label{Lemma:P-ope}
The sets $\mathbf{P}^i d_0$, $\mathbf{P}^i e_0$ and $\mathbf{P}^i \Delta h_1 e_0$ are non-empty for $i\ge 0$ and $j\ge 0$.	\end{lemma}

\begin{proof}Since $d_0$, $e_0$ and $\Delta h_1 e_0$ are generators in Ext, the statement is trivial when $i=0$. 
	
	We now consider the case $i>0$. The Adams periodicity operator $P^i$ is an isomorphism on Ext in specified ranges \cite{Adams-ped}. Since the element $d_0$ lies in these ranges, then $\mathbf{P}^id_0$ contains the element $P^id_0$. Therefore, the set $\mathbf{P}^id_0$ is non-empty.
The same argument is applied for $\mathbf{P}^ie_0$ and $\mathbf{P}^i\Delta h_1 e_0$. 
\end{proof}

\begin{lemma}\label{Lemma:Pg}Let $x$ be an element in $\mathrm{Ext}$ such that $h_1^3\phi(x)=0$. Then $\mathbf{P}^{i+1}\mathbf{g}x$ contains the non-empty set $\mathbf{P}^id_0^2x$ for all $i \ge 0$. As a result, $\mathbf{P}^{i+1}\mathbf{g}x$ is non-empty. 
\end{lemma}
\begin{proof}Since $\mathbf{P}^id_0$ is non-empty by Lemma \ref{Lemma:P-ope}, the set $\mathbf{P}^id_0^2x$ is non-empty. Consider an element $\beta$ in $\mathbf{P}^id_0^2x$. We have 
\begin{align*}	\phi(\beta)=P^id_0^2 \phi(x) &= P^i(Pg+h_1^3\cdot \Delta h_1)\phi(x) \\
&=P^{i+1}g\phi(x)+P^i \Delta h_1 \cdot h_1^3\phi(x)=P^{i+1}g\phi(x). \end{align*}
Consequently, $\mathbf{P}^{i+1}\mathbf{g}x$ contains the element $\beta$ of $\mathbf{P}^id_0^2x$.
\end{proof}

Now we consider $j\ge 2$ and suppose that $j=2^r m$ for $r\ge 0$ and $m$ odd. Classically, we have the following May differential: $d_{2^{r+2}}(P^j)=h_0^5 x_j$ for some $x_j$ in Ext$_{\mathrm{cl}}$. We denote by $\tilde{x}_j$ the motivic element of Chow degree zero corresponding to $x_j$ via the Chow degree zero isomorphism in Theorem \ref{Chow}.

\begin{lemma}\label{lemma:h2g}In $\mathrm{Ext}$, the Massey product $\langle h_2,h_1,h_1^4\tilde{x}_j \rangle$ equals $h_2g^j$.
\end{lemma}

\begin{proof}The motivic elements $h_2g^j, h_2,h_1$ and $h_1^4\tilde{x}_j$ all have Chow degree zero. They correspond to classical elements $P^jh_1$, $h_1$, $h_0$ and $h_0^4 x_j$ via the Chow degree zero isomorphism in Theorem \ref{Chow}.
	
Classically we have $P^jh_1=\langle h_1,h_0,h_0^4x_j \rangle$. We obtain the desired identity by the Chow degree zero isomorphism.  
\end{proof}

\begin{remark}The $g$ in $h_2g^j$ in the above argument is not the operator $\mathbf{g}$. We write $h_2g^j$ for the element of Ext which corresponds to the classical element $P^jh_1$ via the Chow degree zero isomorphism in Theorem \ref{Chow}.  
\end{remark}

\begin{lemma}\label{Lemma:gm}The sets $\mathbf{g}^jm$, $\mathbf{g}^jl$ and $\mathbf{g}^jr$ are non-empty for all $j\ge 0$.
\end{lemma}
\begin{proof}We have $$h_2 \langle h_1,h_1^4\tilde{x}_j,m \rangle = \langle h_2,h_1,h_1^4\tilde{x}_j \rangle \cdot m = h_2g^j \cdot m$$ in which the last identity is by Lemma \ref{lemma:h2g}. Consider an element $\beta$ in $\langle h_1,h_1^4\tilde{x}_j,m \rangle$. We apply $\phi$ to get $$h_2 \phi(\beta) = h_2g^j \phi(m) = h_2n g^{j+1}.$$
	By inspection of Ext$_{\mathbf{A}(2)}$, we have $$\phi(\beta) = ng^{j+1}.$$
  Therefore $\mathbf{g}^jm$ contains $\beta$, and is non-empty.
	
	The same argument is applied to $\mathbf{g}^jl$ and $\mathbf{g}^jr$.
\end{proof}

\begin{lemma}\label{Lemma:tDhg}The set $\tau \Delta h_1 \mathbf{g}^{j+1}$ is non-empty for all $j\ge 0$.
\end{lemma}
\begin{proof}The set $\tau \Delta h_1 \mathbf{g}$ is non-empty since it contains $\tau \Delta h_1 g$.
	
	When $j\ge 1$, the set $\tau \Delta h_1 \mathbf{g}^{j+1} = r \cdot \mathbf{g}^{j-1}m$ is non-empty since $\mathbf{g}^jm$ is non-empty by Lemma \ref{Lemma:gm}. Here we are using the identity $\tau \Delta h_1 g\cdot g=\phi(r) \cdot \phi(m)$ in Ext$_{\mathbf{A}(2)}$.
\end{proof}

\begin{lemma}\label{Lemma:tg}The set
	$\tau \mathbf{g}^j$ is non-empty for any $j\ge 0$. 
\end{lemma}

\begin{proof}The claim for $j=1$ is proven by explicit low dimension calculation \cite{I-stem} \cite{I-MASS}. \\
	 By Lemma \ref{lemma:h2g} $$\langle \tau, h_1^4\tilde{x}_j, h_1 \rangle h_2 = \tau \langle h_1^4\tilde{x}_j,h_1,h_2 \rangle = \tau h_2g^j.$$
Consider an element $\gamma$ in $\langle \tau, h_1^4\tilde{x}_j,h_1 \rangle$, we apply $\phi$ to get
	$$ \phi(\gamma) h_2 = \tau h_2 g^j.$$ 
	By inspection of Ext$_{\mathbf{A}(2)}$, $$ \phi(\gamma) = \tau g^j.$$
	Therefore $\tau \mathbf{g}^j$ contains $\gamma$, and is non-empty. 
\end{proof}

\begin{lemma}\label{Lemma:tPg}The set $\tau \mathbf{P}^i \mathbf{g}^{j+1}$ is non-empty for $i \ge 0$ and $j\ge 0$.
\end{lemma}

\begin{proof}The case $i=0$ is established in Lemma \ref{Lemma:tg}. Now we assume $i>0$.
	When $j=0$, by Lemma \ref{Lemma:P-ope} the set $\mathbf{P}^{i-1}d_0$ is non-empty. Consequently, $\tau \mathbf{P}^{i-1}d_0^2$ is non-empty. We consider an element $x$ in $\tau \mathbf{P}^{i-1}d_0^2$.
	The set $\mathbf{P}^i(\tau g)$ contains $x$ because $$\phi(x)= \tau P^{i-1} d_0^2 = \tau P^{i-1} (Pg + h_1^3 \Delta h_1) = \tau P^ig.$$
	
	When $j\ge 1$, consider $x$ in $\tau \mathbf{g}^{j}$. Since $h_1^3\phi(x)=h_1^3 \cdot \tau g^{j} =0 $ in Ext$_{\mathbf{A}(2)}$, $\mathbf{P}^i\mathbf{g}x$ is non-empty by Lemma \ref{Lemma:Pg}. Therefore $\tau \mathbf{P}^i\mathbf{g}^{j+1}= \mathbf{P}^i\mathbf{g}x$ is non-empty. 
\end{proof}

\begin{ex}\label{Pgt=d}The set $\tau \mathbf{P}\mathbf{g} $ contains $\tau d_0^2$.
\end{ex}

Now we can prove Theorem \ref{def-wedge}.
\begin{proof}	We have the following inclusions of sets.

	 $$\mathbf{P}^i\mathbf{g}^j \tau \Delta h_1 e_0 \supseteq \tau \mathbf{g}^j \cdot \mathbf{P}^i \Delta h_1e_0,$$
		 $$\mathbf{P}^i\mathbf{g}^j e_0r \supseteq \mathbf{P}^ie_0\cdot \mathbf{g}^jr, $$
		$$\mathbf{P}^i\mathbf{g}^j e_0m \supseteq \mathbf{P}^ie_0\cdot \mathbf{g}^jm,$$
		$$\mathbf{P}^i\mathbf{g}^j \tau e_0g^2 \supseteq \mathbf{P}^ie_0\cdot \tau \mathbf{g}^{j+2},$$
		$$\mathbf{P}^i\mathbf{g}^j d_0r \supseteq \mathbf{P}^id_0 \cdot \mathbf{g}^jr,$$
		$$\mathbf{P}^i\mathbf{g}^j d_0m \supseteq \mathbf{P}^id_0\cdot \mathbf{g}^jm,$$
		$$\mathbf{P}^i\mathbf{g}^j \tau e_0^2g \supseteq \mathbf{P}^ie_0^2 \cdot \tau \mathbf{g}^{j+1},$$
		$$\mathbf{P}^i\mathbf{g}^j \tau \Delta h_1e_0^2 \supseteq \tau \mathbf{g}^j \cdot \mathbf{P}^i \Delta h_1 e_0 \cdot e_0,$$
		$$\mathbf{P}^i\mathbf{g}^j d_0l \supseteq \mathbf{P}^id_0\cdot \mathbf{g}^jl,$$
		$$\mathbf{P}^i\mathbf{g}^j \tau e_0^3 \supseteq  \tau \mathbf{g}^j\cdot \mathbf{P}^ie_0^3,$$
		$$\mathbf{P}^i\mathbf{g}^j \tau \Delta h_1d_0e_0 \supseteq \tau \mathbf{g}^j\cdot \mathbf{P}^i\Delta h_1e_0 \cdot d_0,$$
		$$\mathbf{P}^i\mathbf{g}^j d_0er \supseteq \mathbf{P}^jd_0 \cdot \mathbf{g}^jr \cdot e_0.$$
The set $\tau \mathbf{g}^j \cdot \mathbf{P}^i\Delta h_1 e_0$ consists of all products $x\cdot y$ in which $x$ is an element of $\tau \mathbf{g}^j$ and $y$ is an element of $\mathbf{P}^i\Delta h_1 e_0$. The same interpretation is applied for other sets on the right hand side.
 		
The sets on the right hand side are all non-empty because of Lemmas \ref{Lemma:P-ope}, \ref{Lemma:gm} and \ref{Lemma:tg}. For example, since $\tau \mathbf{g}^j$ and $\mathbf{P}^i \Delta h_1e_0$ are non-empty, $$\mathbf{P}^i\mathbf{g}^j \tau \Delta h_1 e_0 \supseteq \tau \mathbf{g}^j \cdot \mathbf{P}^i \Delta h_1e_0$$ is non-empty. Therefore, the sets on the left are all non-empty.

Several values of $\lambda$ remain. 

	Now consider $\lambda = \tau g^2$. The set $\mathbf{P}^i\mathbf{g}^j (\tau g^2)$, or $\tau \mathbf{P}^i \mathbf{g}^{j+2}$, is non-empty by Lemma \ref{Lemma:tPg}. 		
	
	Next consider $\lambda =gr$. The case $i=0$ is established in Lemma \ref{Lemma:gm}. We consider $i>0.$ Since $\mathbf{g}^jr$ is non-empty by Lemma \ref{Lemma:gm}, we consider any element $x$ in $\mathbf{g}^jr$. Since $h_1^3\phi(x)=h_1^3n^2\cdot g^j=0$ \cite{I-A2}, then $\mathbf{P}^i\mathbf{g}^j\lambda =\mathbf{P}^i\mathbf{g}x$ is non-empty by Lemma \ref{Lemma:Pg}. The same argument is applied for $\lambda = gm$.
	
	Finally, consider $\lambda = \tau \Delta h_1g$. The case $i=0$ is established in Lemma \ref{Lemma:tDhg}. We consider $i>0$. When $j>0$, since $\tau \Delta h_1 \mathbf{g}^j$ is non-empty by Lemma \ref{Lemma:tDhg}, we consider any element $x$ in $\tau \Delta h_1 \mathbf{g}^j$. Since $h_1^3\phi(x) = h_1^3\tau \Delta h_1\cdot g^j =0$ \cite{I-A2}, then $\mathbf{P}^i\mathbf{g}^j \lambda = \mathbf{P}^i\mathbf{g} x$ is non-empty by Lemma \ref{Lemma:Pg}. When $j=0$, since
	$$\phi(\tau P^{i-1}d_0\cdot d_0 \Delta h_1)=\tau P^{i-1}d_0^2 \Delta h_1 = \tau P^{i-1} (Pg+h_1^3 \Delta h_1) \Delta h_1 =  P^i \tau \Delta h_1g, $$ 
	 $\mathbf{P}^i\tau \Delta h_1 g$ contains $\tau P^{i-1}d_0\cdot \Delta h_1d_0$, so it is non-empty.
\end{proof}

\section{The $e_0^t\mathbf{g}^k$ and $\Delta h_1 e_0 \mathbf{g}^k$ families}\label{optimize}
Theorem \ref{def-wedge} is not optimal in the sense that there exist elements of weight greater than the weight of elements in $\boldsymbol\Lambda$. For example, the element $\tau e_0^2g$ in $\boldsymbol\Lambda$ is of weight 31. However, the element $e_0^2g$ in Ext is of weight 32. Table \ref{opt} lists all such elements in $\boldsymbol\Lambda$. 

\begin{table}[h]
	\caption{\label{opt}}
	\begin{tabular}{c|c|c|c}
		
		element of the wedge& weight & element of higher weight& weight \\
		\hline
		$\tau \Delta h_1 e_0$ & 22 &  $ \Delta h_1 e_0$ & 23	  \\
		
		$\tau e_0^3$ & 29 & $e_0^3$ & 30  \\
		
		 $\tau \Delta h_1 d_0 e_0$ & 30 & $\Delta h_1 d_0 \cdot e_0$ & 31 \\
		 
		$\tau e_0^2g$ & 31 & $e_0\cdot e_0g$ & 32  \\
		
		$\tau \Delta h_1 e_0^2$ & 32 &  $ \Delta h_1 e_0\cdot e_0$ & 33	    
\end{tabular} \end{table}

\begin{remark}By $\mathbf{g}^j$ we mean $\mathbf{g}^j(1)$ which is understood in the sense of Definition \ref{def:Pg}. 
\end{remark}
\begin{lemma}\label{g-empty}The set $\mathbf{g}^j$ is empty for all $j\ge 0$. \end{lemma}

\begin{proof}We prove the statement via contradiction. Suppose that $\mathbf{g}^j$ is non-empty. Consider any element $x$ in $\mathbf{g}^j$. Since $x$ maps to the non-zero element $g^j$ in Ext$_{\mathbf{A}(2)}$, $x$ is non-zero. Furthermore, because $x$ has Chow degree zero, $x$ corresponds to a classical element at degree $(8j,4j)$ in Ext$_{\mathrm{cl}}$ via the Chow degree zero isomorphism. However, Ext$_{\mathrm{cl}}$ is zero in degrees $(8j,4j)$ for all non-negative integers $j$ \cite{Adams-ped}. Therefore, $x$ does not exist. 
\end{proof}

\begin{lemma}\label{gd}The set $d_0\mathbf{g}$ contains $e_0^2$.
	\end{lemma}
\begin{proof}We have $\phi (e_0^2)=e_0^2 = gd.$ 
	The last identity is because $e_0^2 = gd$ in Ext$_{\mathbf{A}(2)}.$
	\end{proof}
We study the behavior of the sets $\tau \mathbf{P}^i\mathbf{g}^j \tau e_0^3$ for $i\ge 0$ and $j\ge 0$. 

\begin{thm}\label{thm:ij}The set $\tau \mathbf{P}^i \mathbf{g}^j e_0^3$ contains an element divisible by $\tau$ if
	\begin{itemize}
		\item $i\ge 0$ and $j=0$, or
		\item $i\ge j\ge 1$, or
		\item $1\le i < j \le 3i$.
		\end{itemize}
	\end{thm}
\begin{proof}When $i=0$ and $j=0$, the element $\tau e_0^3$ is divisible by $\tau$. When $i\ge 1$ and $j=0$, the set $\tau \mathbf{P}^ie_0^3$ contains the element $\tau \cdot P^ie_0^3$ which is divisible by $\tau.$ 
	Apply Example \ref{Pgt=d} and Lemma \ref{gd} to get:
	\begin{itemize}\item When $i\ge j \ge 1$, the set $\tau \mathbf{P}^i\mathbf{g}^j e_0^3$ contains the element $\tau \cdot d_0^{2j} \cdot P^{i-j}e_0^3$ which is divisible by $\tau$.
		\item When $1\le i < j \le 3i$, the set $\tau \mathbf{P}^i\mathbf{g}^j  e_0^3$ contains the element $\tau \cdot d_0^{3i-j} \cdot e_0^{2(j-i)+3}$ which is divisible by $\tau.$
	\end{itemize}
	\end{proof}
There are unknown cases from Theorem \ref{thm:ij}. When $i=0$ and $j\ge 1$, the set $\tau e_0^3 \mathbf{g}^j $ is not known fully and will be of our interest. 
When $3i<j$, the set $\mathbf{P}^i\mathbf{g}^j \tau e_0^3$ contains the set $\tau e_0^{4i+3}\mathbf{g}^{j-3i}$ which is not known fully and will be of our interest.

We apply the same argument for the sets $\mathbf{P}^i\mathbf{g}^j \tau \Delta h_1e_0$, $\mathbf{P}^i\mathbf{g}^j \tau \Delta h_1 d_0 e_0$, $\mathbf{P}^i\mathbf{g}^j \tau e_0^2 g$ and $\mathbf{P}^i\mathbf{g}^j \tau \Delta h_1 e_0^2$ for $i\ge 0$ and $j\ge 0$ to observe that we need to study the behavior of the families of sets $\Delta h_1 e_0^t \mathbf{g}^j $ and $e_0^t \mathbf{g}^j $ for $i \ge 0$, $j \ge 0$ and $t\ge 1$.

\begin{remark}Since $\Delta h_1$ is not an element of Ext, $\Delta h_1 \mathbf{g}^j$ is not defined in Ext. Therefore, we do not consider the set $\Delta h_1 e_0^t \mathbf{g}^j$ when $t=0$.
	\end{remark}
\subsection{The $e_0^t\mathbf{g}^k$ family}

%For $e_0^2g, e_0g^2$ and $e_0^3$, by explicit computations, the answer is complicated. Sometimes it is yes, sometimes no. It leads us to study the elements of the form $e_0^tg^k$.

\begin{lemma}
	\label{lem:h1-local}
	If $e_0^t\mathbf{g}^k$ is non-empty, then $e_0^t\mathbf{g}^k$ consists of elements which are non-zero in the $h_1$-localization $\mathrm{Ext}[h_1^{-1}]$.
\end{lemma}
\begin{proof}For any element $x$ in $e_0^t\mathbf{g}^k$ and any non-negative integer $n$, we have $\phi (h_1^n x)=h_1^n e_0^tg^k$ which is non-zero in Ext$_{\mathbf{A}(2)}$ \cite{I-A2}. Consequently, $h_1^n x$ is non-zero in Ext. In other words, $x$ is non-zero in the $h_1$-localization Ext$_{\mathbf{A}}[h_1^{-1}]$. \end{proof}

\begin{prop}\label{prop:eg-empty}Let $t$ and $k$ be non-negative integers. If $\alpha (t+k) >t $, then $e_0^t \mathbf{g}^k$ is empty.
\end{prop}

\begin{proof}(Via contradiction) Suppose that $e_0^{t}\mathbf{g}^k$ is non-empty. As a result, its elements survive the $h_1$-localization by Lemma \ref{lem:h1-local}. Note that elements of $e_0^{t}\mathbf{g}^k$ have Chow degree $t$ and coweight $(7t+8k)$. By Theorem \ref{thm:h1-local}, after considering Chow degrees, any element of $e_0^{t}\mathbf{g}^k$ maps to a summation of monomials of the form  $$v_1^{4n} v_2^m \prod_{i=1}^{t-4n-m}v_{m_i}$$ in Ext$_{\mathbf{A}}[h_1^{-1}]$ for some $n, m$ and $m_i\ge 3$. By comparing coweights, we have
	$$7t+8k = 4n +3m + \sum_{i=1}^{t-4n-m} (2^{m_i}-1).$$
	Then $$8t+8k = 8n + 4m + \sum_{i=1}^{t-4n-m} 2^{m_i}.$$
	Since $m_i \ge 3$, $m$ has to be even, i.e., $m=2m'$ for some non-negative integer $m'$.
	We obtain $$t+k=n+m'+\sum_{i=1}^{t-4n-m}2^{m_i-3}.$$ By Lemma \ref{lem:binary}, $$\alpha (t+k) \le \alpha (n) + \alpha (m') + t - 4n - m = t + (\alpha (n) - 4n) + (\alpha (m') - 2m') \le t.$$
\end{proof}

\begin{cor}\label{cor:empty}If $e_0 \mathbf{g}^k$ is non-empty, then $k=2^n-1$ for some non-negative integer $n$.
\end{cor}
\begin{proof}Since $e_0 \mathbf{g}^k$ is non-empty, $\alpha (1+k) \le 1$ by Proposition \ref{prop:eg-empty}. Then $1+k=2^n$ for some non-negative integer $n$.   
\end{proof}
We state the following conjecture.

\begin{conj}\label{conj:eg}The set $e_0\mathbf{g}^k$ is non-empty if and only if $k=2^n-1$ for some non-negative integer $n$.
\end{conj}

We mention some evidence supporting the conjecture. The elements $e_0g$ and $e_0g^3$ survive in Ext (by explicit computations). Also, the conjecture fits nicely with the properties of the $h_1$-localization of Ext \cite{G-I-local}.

\begin{thm}\label{eg}Suppose that $e_0\mathbf{g}^{2^n-1}$ is non-empty for every non-negative integer $n$. Then $e_0^t\mathbf{g}^k$ is non-empty if and only if $k = (\sum_{i=1}^{t} 2^{n_i}) - t $ for some non-negative integers $n_i$.
\end{thm}
\begin{proof}If $e_0^t\mathbf{g}^k$ is non-empty, then by Proposition \ref{prop:eg-empty} we have $\alpha(k+t)\le t$. As a result, $k + t = \sum_{i=1}^{t} 2^{n_i}$ for some non-negative integers $n_i$. In other words,  $$k = (\sum_{i=1}^{t} 2^{n_i}) - t .$$ 
	
	Conversely, if $k = (\sum_{i=1}^{t} 2^{n_i}) - t $, since $e_0\mathbf{g}^{2^{n_i}-1}$ is non-empty for all $n_i$ then $$e_0^t\mathbf{g}^k \supseteq e_0\mathbf{g}^{2^{n_1}-1}  \cdots  e_0\mathbf{g}^{2^{n_t}-1}$$ is non-empty.
\end{proof}

\begin{remark}\label{rem:eg}The condition $k = (\sum_{i=1}^{t} 2^{n_i}) - t $ is equivalent to $\alpha(k+t) \le t$. In practice, we use the latter condition rather than the former one. 
	\end{remark}

\subsection{The $\Delta h_1e_0^t\mathbf{g}^k$ family}

\begin{prop}{\label{dheg}}If $\alpha (1+k+t)>t$ for $t\ge 1$ and $k\ge 0$, then the set $\Delta h_1e_0^t\mathbf{g}^k$ is empty. 
\end{prop}
\begin{proof}(Via contradiction) We recall the following commutative diagram \cite{G-I-local}
	\[
	\begin{tikzcd}
	\mathrm{Ext}  \ar{r}{\phi} \arrow{d}{L} & 
	\mathrm{Ext}_{\mathbf{A}(2)} \arrow{d}{L}  \\
	\mathrm{Ext}[h_1^{-1}] \ar{r}{\phi} & \mathrm{Ext}_{\mathbf{A}(2)}[h_1^{-1}] 
	\end{tikzcd}
	\] 
	Suppose that $\Delta h_1e_0^t\mathbf{g}^k$ is non-empty. Then it contains an element $x$. The element $x$ maps to the element $\Delta h_1e_0^tg^k$ in Ext$_{\mathbf{A}(2)}$, surviving $h_1$-localization. The element $\Delta h_1e_0^tg^k$ maps to 
	$$h_1^{-2k-5} v_1^4 a_1^{2+2k+t} v_2^t + h_1^{-2k+1} v_2^{4+t} a_1^{2k+t}$$ in Ext$_{\mathbf{A}(2)}[h_1^{-1}]$ via $L$.
	
	Since $\alpha (1+k+t)>t$, the term $h_1^{-2k-5} v_1^4 a_1^{2+2k+t} v_2^t$ is not in the image of $\phi:$ Ext$[h_1^{-1}] \longrightarrow$ Ext$_{\mathbf{A}(2)}[h_1^{-1}]$  by Lemma \ref{local2}.
\end{proof}

\begin{lemma}For any integer $k\ge 0$, there is no element $x$ in $\mathrm{Ext}$ such that $\phi(x) = \Delta h_1 g^k$ in $\mathrm{Ext}_{\mathbf{A}(2)}$.
	\end{lemma}
\begin{proof}We apply the same argument as in Proposition \ref{dheg}.
	\end{proof}
%	We prove the lemma via contradiction. If $\Delta h_1 g^j$ is in the image of $\phi$, then it maps to the element  $$h_1^{-2j-5}v_1^4a_1^{2+2j}+h_1^{-2j+1}v_2^4a_1^{2j}$$ in Ext$_{\mathbf{A}(2)}[h_1^{-1}]$ via $L$. However, since $\alpha (2+2j)>0$, then $h_1^{-2j-5}v_1^4a_1^{2+2j}$ is not in the image of $\phi:$ Ext$_{\mathbf{A}}[h_1^{-1}] \longrightarrow$ Ext$_{\mathbf{A}(2)}[h_1^{-1}]$  by Lemma \ref{local2}.

 By Proposition \ref{dheg}, a necessary condition for the set $\Delta h_1e_0\mathbf{g}^j$ to be non-empty is $\alpha (2+k)\le 1$, or $k=2^n-2$ for some non-negative integer $n$. Unfortunately, we do not know if it is sufficient. We state the following conjecture.
\begin{conj}\label{delta-eg}The set $\Delta h_1e_0\mathbf{g}^j$ is non-empty if and only if $k=2^n-2$ for some non-negative integer $n$. 
\end{conj}

\begin{thm}\label{thm:delta-eg}Suppose that $e_0\mathbf{g}^{2^n-1}$ and $\Delta h_1e_0\mathbf{g}^{2^n-2}$ are non-empty for every non-negative integer $n$. Then $\Delta h_1 e_0^t\mathbf{g}^k$ is non-empty if and only if $k = (\sum_{i=1}^{t} 2^{n_i}) - t -1$ for some non-negative integers $n_i$.
\end{thm}
\begin{proof}If $\Delta h_1 e_0^t\mathbf{g}^k$ is non-empty, then $\alpha (1+k+t)\le t$ or $k = (\sum_{i=1}^{t} 2^{n_i}) - t -1$ for some non-negative integers $n_i$.
	
	Conversely, if $k = (\sum_{i=1}^{t} 2^{n_i}) - t -1$ for some non-negative integers $n_i$, then $\Delta h_1 e_0^t\mathbf{g}^k$ contains the set
	$$\Delta h_1 e_0 \mathbf{g}^{2^{n_{i_1}}-2} e_0 \mathbf{g}^{2^{n_{i_2}}-1} \cdot \ldots e_0  \mathbf{g}^{2^{n_{i_t}}-1}$$ which is non-empty.
\end{proof}

\subsection{The wedge at filtrations $f=4k$ and $f=4k+1$ for $k\ge 2$}
At filtrations $f=4k+2$ and $f=4k+3$ for $k\ge 2$, the wedge is known completely. At filtrations $f=4k$ and $f=4k+1$ for $k\ge 2$, the wedge remains not optimal. If Conjectures \ref{conj:eg} and \ref{delta-eg} are correct, we can completely solve the problem of optimizing the wedge mentioned at the beginning of section \ref{optimize}.

%\begin{remark}We recall that the Adams periodicity operator $P^i$ is an isomorphism on Ext in specified ranges \cite{Adams-ped}. Since the elements $d_0$, $e_0$ and $\Delta h_1 e_0$ lie in these ranges, then $P^id_0$, $P^i e_0$ and $P^i \Delta h_1 e_0$ are elements of Ext for all $i\ge 0$.
	%\end{remark}
The wedge contains the following elements and sets (the order is of increasing stems) at filtrations $f=4k$ for $k\ge 3$ 
	$$\tau \cdot P^{k-3}d_0 e_0^2, \tau \cdot P^{k-3} e_0^3, \tau \cdot P^{k-4} e_0^{2} d_0^2, \tau \cdot P^{k-4} e_0^{3} d_0, \tau \cdot P^{k-4} e_0^{4}, \ldots,  $$
	$$\tau \cdot e_0^{k-2}d_0^2, \tau  \cdot e_0^{k-1}d_0, e_0^k, \tau e_0^{k-1}\mathbf{g}, \ldots, \tau e_0 \mathbf{g}^{k-1}, \tau \mathbf{g}^k.$$

 \begin{thm}
 Suppose that $e_0\mathbf{g}^{2^n-1}$ is non-empty for every non-negative integer $n$. Then at filtration $f=4k$ for $k\ge 2$ the set $\tau e_0^s \mathbf{g}^{k-s}$ contains an element divisible by $\tau$ if $s\ge \alpha(k)$ and does not contain any element divisible by $\tau$ if $s < \alpha (k)$.
 \end{thm}
\begin{proof}If $s\ge \alpha(k)$, by Theorem \ref{eg} and Remark \ref{rem:eg} the set $e_0^s \mathbf{g}^{k-s}$ contains an element $x$. Then $\tau e_0^s \mathbf{g}^{k-s}$ contains the element $\tau \cdot x$ divisible by $\tau$.
	
	If $s< \alpha (k)$, we suppose that $\tau e_0^s \mathbf{g}^{k-s}$ contains an element $\tau \cdot y$ divisible by $\tau$. The element $\tau \cdot y$ maps to $\tau e_0^s g^{k-s}$ in Ext$_{\mathbf{A}(2)}$. Then $y$ maps to $e_0^s g^{k-s}$ in Ext$_{\mathbf{A}(2)}$. In other words, $y$ is an element of the set $e_0^s \mathbf{g}^{k-s}$. However, since $s< \alpha (k)$, the set $e_0^s \mathbf{g}^{k-s}$ is empty by Proposition \ref{prop:eg-empty}.
	\end{proof}

The wedge contains the following elements and sets (the order is of increasing stems) at the filtration $f=4k+1$ for $k\ge 2$
$$\tau \cdot P^{k-2}\Delta h_1 e_0, \tau \cdot P^{k-3} d_0 \Delta h_1 d_0, \tau \cdot P^{k-3} d_0 \Delta h_1 e_0, \tau \cdot P^{k-3} \Delta h_1 e_0^2, \ldots $$
$$\tau \cdot d_0^2 \Delta h_1 e_0^{k-3}, \tau \cdot d_0 \Delta h_1 e_0^{k-2}, \tau \cdot \Delta h_1 e_0^{k-1}, \tau \Delta h_1 e_0^{k-2} \mathbf{g}, \ldots, \tau \Delta h_1 e_0 \mathbf{g}^{k-2}, \tau \Delta h_1 \mathbf{g}^{k-1}. $$

\begin{thm}
Suppose that $e_0\mathbf{g}^{2^n-1}$ and $\Delta h_1e_0\mathbf{g}^{2^n-2}$ are non-empty for every non-negative integer $n$. Then at filtration $f=4k+1$ for $k\ge 2$ the set $\tau \Delta h_1 e_0^{s} \mathbf{g}^{k-s-1}$ contains an element divisible by $\tau$ if $s\ge \alpha(k)$ and does not contain any element divisible by $\tau$ if $s< \alpha(k)$.
\end{thm}
\begin{proof}If $s\ge \alpha(k)$, then by Theorem \ref{thm:delta-eg} the set $\Delta h_1 e_0^{s} \mathbf{g}^{k-s-1}$ contains an element $x$. Then $\tau \Delta h_1 e_0^{s} \mathbf{g}^{k-s-1}$ contains the element $\tau \cdot x$ divisible by $\tau$.
	
	If $s< \alpha(k)$, we suppose that $\tau \Delta h_1 e_0^{s} \mathbf{g}^{k-s-1}$ contains an element $\tau \cdot y$ divisible by $\tau$. The element $\tau \cdot y$ maps to $\tau \Delta h_1 e_0^s g^{k-s-1}$ in Ext$_{\mathbf{A}(2)}$. Then the element $y$ maps to $\Delta h_1 e_0^s g^{k-s-1}$ in Ext$_{\mathbf{A}(2)}$. In other words, $y$ is an element of the set $\Delta h_1 e_0^s \mathbf{g}^{k-s-1}$. However, since $s< \alpha (k)$, the set $\Delta h_1 e_0^s \mathbf{g}^{k-s-1}$ is empty by Proposition \ref{dheg}.
	\end{proof}

\subsection{The wedge chart}
This chart shows the wedge from its vertex to stem 70.
\begin{itemize}
	\item All dots indicate copies of $\mathbb{M}_2$.
	\item Red dots indicate elements which behave irregularly, as in Propositions \ref{prop:eg-empty} and \ref{dheg}. 
	\end{itemize} 

\begin{landscape}
\newrgbcolor{cldotcolor}{0.5 0.5 0.5}
\newrgbcolor{clhzerocolor}{0.5 0.5 0.5}
\newrgbcolor{clhonecolor}{0.5 0.5 0.5}
\newrgbcolor{clhtwocolor}{0.5 0.5 0.5}

\newrgbcolor{Ctauhiddenhzerocolor}{0 0 1}
\newrgbcolor{Ctauhiddenhonecolor}{0 0 1}
\newrgbcolor{Ctauhiddenhtwocolor}{0 0 1}

\newrgbcolor{taubottomcolor}{0.5 0.5 0.5}
\newrgbcolor{tautopcolor}{1 0 0}
\newrgbcolor{taubottomhidcolor}{0.2 0.8 0.2}
\newrgbcolor{tautophidcolor}{1 0.6 0.2}
\newrgbcolor{tauunknowncolor}{0.7 0.1 0.8}

\newrgbcolor{tauzerocolor}{0.5 0.5 0.5}
\newrgbcolor{tauonecolor}{1 0 0}
\newrgbcolor{tautwocolor}{0.3 0.3 1}
\newrgbcolor{tauthreecolor}{0 0.7 0}
\newrgbcolor{taufourcolor}{0.7 0.1 0.8}
\newrgbcolor{taufivecolor}{0.7 0.1 0.8}
\newrgbcolor{tausixcolor}{0.7 0.1 0.8}

\newrgbcolor{hzerotaucolor}{1 0 1}
\newrgbcolor{hzeromoretaucolor}{1 0.5 0}
\newrgbcolor{hzerotowercolor}{0.5 0.5 0.5}

\newrgbcolor{honetaucolor}{1 0 1}
\newrgbcolor{honemoretaucolor}{1 0.5 0}
\newrgbcolor{honetowercolor}{1 0 0}

\newrgbcolor{htwotaucolor}{1 0 1}
\newrgbcolor{htwomoretaucolor}{1 0.5 0}

\newrgbcolor{dtwocolor}{0 0.7 0.7}
\newrgbcolor{dtwotaucolor}{1 0 1}
\newrgbcolor{dtwomoretaucolor}{1 0.5 0}
\newrgbcolor{dthreecolor}{1 0 0}
\newrgbcolor{dfourcolor}{0.1 0.7 0.1}
\newrgbcolor{dfivecolor}{0.2 0.2 0.7}
\newrgbcolor{dsixcolor}{1 0.5 0}
\newrgbcolor{dsevencolor}{1 0.5 0}
\newrgbcolor{deightcolor}{1 0.5 0}

\newrgbcolor{tauextncolor}{0 0.6 0} 
\newrgbcolor{twoextncolor}{0.7 0.7 0}
\newrgbcolor{etaextncolor}{0.5 0 1}
\newrgbcolor{nuextncolor}{0.6 0.3 0}

\newgray{gridline}{0.8}
\newgray{unknowncolor}{0.95}

\newcommand{\cirrad}{0.06}
\newcommand{\dotseplength}{0.05}

\newpsobject{tauextn}{psline}{linecolor=tauextncolor}
\newpsobject{tauextncurve}{psbezier}{linecolor=tauextncolor}
\newpsobject{twoextn}{psline}{linecolor=twoextncolor}
\newpsobject{twoextncurve}{psbezier}{linecolor=twoextncolor}
\newpsobject{etaextn}{psline}{linecolor=etaextncolor}
\newpsobject{nuextn}{psline}{linecolor=nuextncolor}
\newpsobject{nuextncurve}{psbezier}{linecolor=nuextncolor}
\newpsobject{Ctautwoextncurve}{psbezier}{linecolor=Ctauhiddenhzerocolor}
\newpsobject{Ctauetaextncurve}{psbezier}{linecolor=Ctauhiddenhonecolor}
\newpsobject{Ctaunuextncurve}{psbezier}{linecolor=Ctauhiddenhtwocolor}

\newcommand{\D}{\Delta}

\newcommand{\myindent}{1ex}

% 1.0x scale
\psset{unit=0.71cm}

%% 1.5x scale
\psset{unit=0.65cm}
\renewcommand{\cirrad}{0.1}
%\renewcommand{\dotseplength}{0.033}
%
%\psset{linewidth=0.3mm}
 \begin{figure}
\begin{center}
\begin{pspicture}(39,7)(70,25)
\tiny
\psgrid[unit=2,gridcolor=gridline,subgriddiv=0,gridlabelcolor=white](20,4)(35,12)

\rput(40,7){40}
\rput(42,7){42}
\rput(44,7){44}
\rput(46,7){46}
\rput(48,7){48}
\rput(50,7){50}
\rput(52,7){52}
\rput(54,7){54}
\rput(56,7){56}
\rput(58,7){58}
\rput(60,7){60}
\rput(62,7){62}
\rput(64,7){64}
\rput(66,7){66}
\rput(68,7){68}
\rput(70,7){70}

\rput(-1,0){0}
\rput(-1,2){2}
\rput(-1,4){4}
\rput(-1,6){6}
\rput(39,8){8}
\rput(39,10){10}
\rput(39,12){12}
\rput(39,14){14}
\rput(39,16){16}
\rput(39,18){18}
\rput(39,20){20}
\rput(39,22){22}
\rput(39,24){24}
\rput(-1,26){26}
\rput(-1,28){28}
\rput(-1,30){30}
\rput(-1,32){32}
\rput(-1,34){34}
\rput(-1,36){36}
\rput(-1,38){38}
\rput(-1,40){40}
\rput(-1,42){42}
\rput(-1,44){44}
\rput(-1,46){46}
\rput(-1,48){48}
\rput(-1,50){50}
\rput(-1,52){52}
\rput(-1,54){54}
\rput(-1,56){56}

\pscircle*[linecolor=tauonecolor](40.00,8){\cirrad}
\uput{\cirrad}[-90](40.00,8){$\tau g^2$}
\pscircle*[linecolor=tauzerocolor](42.00,9){\cirrad}
\uput{\cirrad}[-90](42.00,9){$ \D h_1 e_0$}
\pscircle*[linecolor=tauzerocolor](44.00,10){\cirrad}
\uput{\cirrad}[-90](44.00,10){$d_0 r$}
\pscircle*[linecolor=tauonecolor](45.00,9){\cirrad}
\uput{\cirrad}[-90](45.00,9){$\tau \D h_1 g $}
\pscircle*[linecolor=tauzerocolor](46.00,11){\cirrad}
\uput{\cirrad}[150](46.00,11){$d_0 l$}
\pscircle*[linecolor=tauzerocolor](47.00,10){\cirrad}
\uput{\cirrad}[-90](47.00,10){$e_0 r$}
\pscircle*[linecolor=tauzerocolor](48.00,12){\cirrad}
\uput{\cirrad}[-90](48.00,12){$ d_0 e_0^2$}
\pscircle*[linecolor=tauzerocolor](49.00,11){\cirrad}
\uput{\cirrad}[150](49.00,11){$d_0 m$}
\pscircle*[linecolor=tauzerocolor](50.00,10){\cirrad}
\uput{\cirrad}[-90](50.00,10){$gr$}
\pscircle*[linecolor=tauzerocolor](50.00,13){\cirrad}
\uput{\cirrad}[-90](50.00,13){$ P \D h_1 e_0$}
\pscircle*[linecolor=tauzerocolor](51.00,12){\cirrad}
\uput{\cirrad}[-90](51.00,12){$ e_0^3$}
\pscircle*[linecolor=tauzerocolor](52.00,11){\cirrad}
\uput{\cirrad}[-90](52.00,11){$e_0 m$}
\pscircle*[linecolor=tauzerocolor](52.00,14){\cirrad}
\uput{\cirrad}[-90](52.00,14){$P d_0 r$}
\pscircle*[linecolor=tauzerocolor](53.00,13){\cirrad}
\uput{\cirrad}[-90](53.00,13){$\Delta h_1 d_0^2  $}
\pscircle*[linecolor=tauzerocolor](54.00,12){\cirrad}
\uput{\cirrad}[-90](54.00,12){$ e_0^2 g$}
\pscircle*[linecolor=tauzerocolor](54.00,15){\cirrad}
\uput{\cirrad}[150](54.00,15){$P d_0 l$}
\pscircle*[linecolor=tauzerocolor](55.00,11){\cirrad}
\uput{\cirrad}[-90](55.00,11){$gm$}
\pscircle*[linecolor=tauzerocolor](55.00,14){\cirrad}
\uput{\cirrad}[-90](55.00,14){$P e_0 r$}
\pscircle*[linecolor=tauzerocolor](56.00,13){\cirrad}
\uput{\cirrad}[-60](56.00,13){$ \D h_1 d_0 e_0$}
\pscircle*[linecolor=tauzerocolor](56.00,16){\cirrad}
\uput{\cirrad}[-90](56.00,16){$ P d_0 e_0^2$}
\pscircle*[linecolor=tauonecolor](57.00,12){\cirrad}
\uput{\cirrad}[-90](57.00,12){$\tau e_0 g^2$}
\pscircle*[linecolor=tauzerocolor](57.00,15){\cirrad}
\uput{\cirrad}[150](57.00,15){$P d_0 m$}
\pscircle*[linecolor=tauzerocolor](58.00,14){\cirrad}
\uput{\cirrad}[-90](58.00,14){$d_0^2 r$}
\pscircle*[linecolor=tauzerocolor](58.00,17){\cirrad}
\uput{\cirrad}[-90](58.00,17){$ P^2 \D h_1 e_0$}
\pscircle*[linecolor=tauzerocolor](59.00,13){\cirrad}
\uput{\cirrad}[-90](59.00,13){$ \D h_1 e_0^2$}
\pscircle*[linecolor=tauzerocolor](59.00,16){\cirrad}
\uput{\cirrad}[-90](59.00,16){$ P e_0^3$}
\pscircle*[linecolor=tauonecolor](60.00,12){\cirrad}
\uput{\cirrad}[-90](60.00,12){$\tau g^3$}
\pscircle*[linecolor=tauzerocolor](60.00,15){\cirrad}
\uput{\cirrad}[150](60.00,15){$P e_0 m$}
\pscircle*[linecolor=tauzerocolor](60.00,18){\cirrad}
\uput{\cirrad}[-90](60.00,18){$P^2 d_0 r$}
\pscircle*[linecolor=tauzerocolor](61.00,14){\cirrad}
\uput{\cirrad}[-90](61.00,14){$d_0 e_0 r$}
\pscircle*[linecolor=tauzerocolor](61.00,17){\cirrad}
\uput{\cirrad}[-90](61.00,17){$ P \Delta h_1  d_0^2 $}
\pscircle*[linecolor=tauonecolor](62.00,13){\cirrad}
\uput{\cirrad}[-90](62.00,13){$\tau \D h_1 e_0 g$}
\pscircle*[linecolor=tauzerocolor](62.00,16){\cirrad}
\uput{\cirrad}[-120](62.00,16){$ P e_0^2 g$}
\pscircle*[linecolor=tauzerocolor](62.00,19){\cirrad}
\uput{\cirrad}[150](62.00,19){$P^2 d_0 l$}
\pscircle*[linecolor=tauzerocolor](63.00,15){\cirrad}
\uput{\cirrad}[150](63.00,15){$d_0^2 m$}
\pscircle*[linecolor=tauzerocolor](63.00,18){\cirrad}
\uput{\cirrad}[-120](63.00,18){$P^2 e_0 r$}
\pscircle*[linecolor=tauzerocolor](64.00,14){\cirrad}
\uput{\cirrad}[90](64.00,14){$d_0 g r$}
\pscircle*[linecolor=tauzerocolor](64.00,17){\cirrad}
\uput{\cirrad}[-90](64.00,17){$ P \D h_1 d_0 e_0$}
\pscircle*[linecolor=tauzerocolor](64.00,20){\cirrad}
\uput{\cirrad}[-90](64.00,20){$ P^2 d_0 e_0^2$}
\pscircle*[linecolor=tauonecolor](65.00,13){\cirrad}
\uput{\cirrad}[90](65.00,13){$\tau \D h_1 g^2$}
\pscircle*[linecolor=tauzerocolor](65.00,16){\cirrad}
\uput{\cirrad}[-90](65.00,16){$ d_0^2 e_0g$}
\pscircle*[linecolor=tauzerocolor](65.00,19){\cirrad}
\uput{\cirrad}[150](65.00,19){$P^2 d_0 m$}
\pscircle*[linecolor=tauzerocolor](66.00,15){\cirrad}
\uput{\cirrad}[150](66.00,15){$d_0 e_0 m$}
\pscircle*[linecolor=tauzerocolor](66.00,18){\cirrad}
\uput{\cirrad}[-90](66.00,18){$P d_0^2 r$}
\pscircle*[linecolor=tauzerocolor](66.00,21){\cirrad}
\uput{\cirrad}[-90](66.00,21){$ P^3 \D h_1 e_0$}
\pscircle*[linecolor=tauzerocolor](67.00,14){\cirrad}
\uput{\cirrad}[90](67.00,14){$e_0 g r$}
\pscircle*[linecolor=tauzerocolor](67.00,17){\cirrad}
\uput{\cirrad}[-90](67.00,17){$ P \D h_1 e_0^2$}
\pscircle*[linecolor=tauzerocolor](67.00,20){\cirrad}
\uput{\cirrad}[-90](67.00,20){$ P^2 e_0^3$}
\pscircle*[linecolor=tauzerocolor](68.00,16){\cirrad}
\uput{\cirrad}[-90](68.00,16){$ e_0^4$}
\pscircle*[linecolor=tauzerocolor](68.00,19){\cirrad}
\uput{\cirrad}[150](68.00,19){$P^2 e_0 m$}
\pscircle*[linecolor=tauzerocolor](68.00,22){\cirrad}
\uput{\cirrad}[-90](68.00,22){$P^3 d_0 r$}
\pscircle*[linecolor=tauzerocolor](69.00,15){\cirrad}
\uput{\cirrad}[-90](69.00,15){$e_0^2 m$}
\pscircle*[linecolor=tauzerocolor](69.00,18){\cirrad}
\uput{\cirrad}[-90](69.00,18){$P d_0 e_0 r$}
\pscircle*[linecolor=tauzerocolor](69.00,21){\cirrad}
\uput{\cirrad}[-90](69.00,21){$ P^2 \D h_1 d_0^2$}
\pscircle*[linecolor=tauzerocolor](70.00,14){\cirrad}
\uput{\cirrad}[0](70.00,14){$m^2$}
\pscircle*[linecolor=tauzerocolor](70.00,17){\cirrad}
\uput{\cirrad}[-90](70.00,17){$ \D h_1 d_0^2 e_0 $}
\pscircle*[linecolor=tauzerocolor](70.00,20){\cirrad}
\uput{\cirrad}[-90](70.00,20){$ P^2 e_0^2 g$}
\pscircle*[linecolor=tauzerocolor](70.00,23){\cirrad}
\uput{\cirrad}[150](70.00,23){$P^3 d_0 l$}

\end{pspicture}

\end{center}
\caption{The $\mathbb{C}$-motivic wedge through the 70-stem}
\label{figure}
\end{figure}
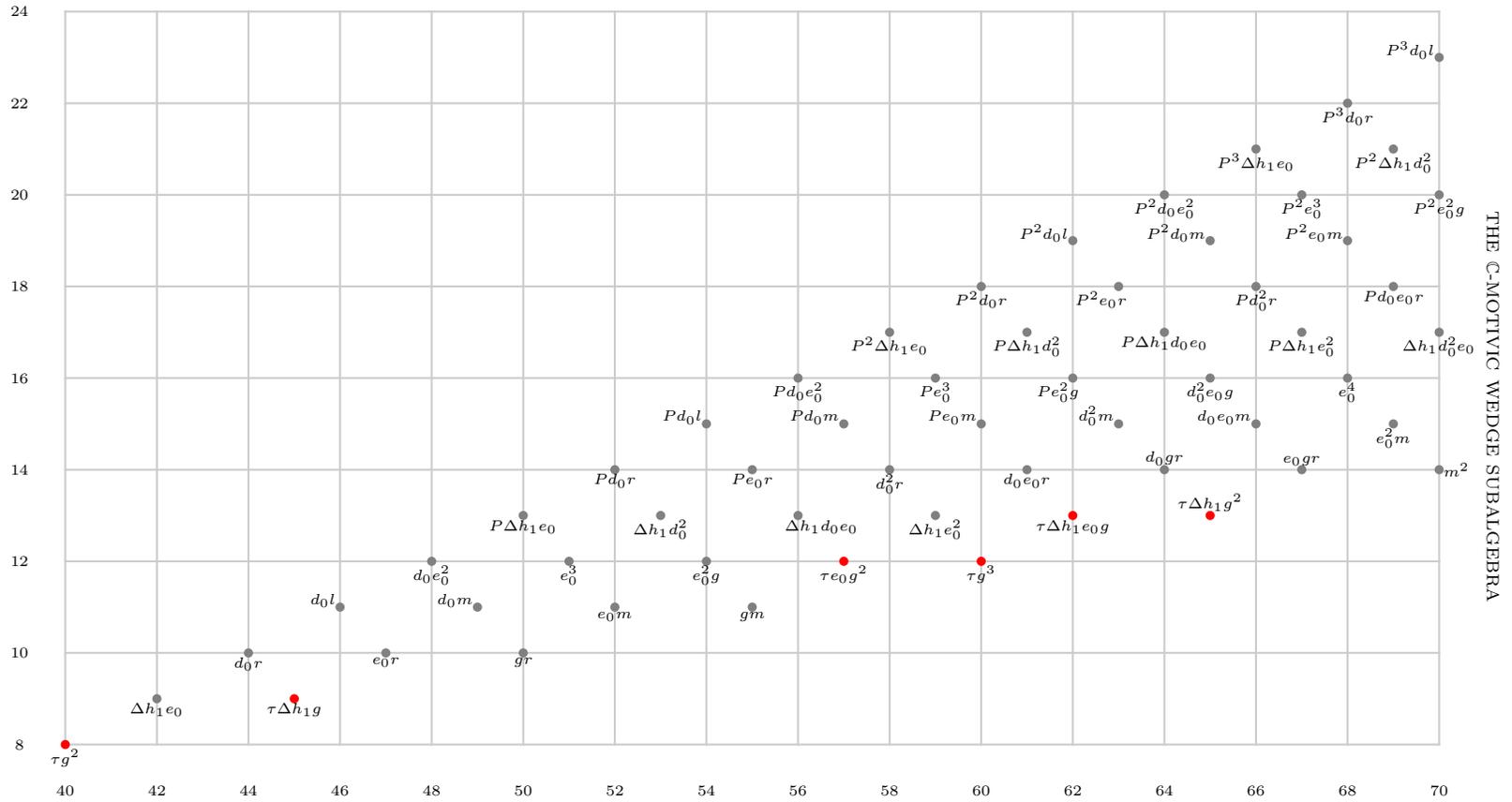

\end{landscape}

%\bibliographystyle{amsalpha}
%\begin{bibdiv}
%\begin{biblist}
%
%\bibselect{stable-stems-bib}
%
%\end{biblist}
%\end{bibdiv}

\end{document}